\documentclass[final,12pt]{elsarticle}

\usepackage{amssymb}
\usepackage{amsmath,amssymb,amscd,amsthm,verbatim,alltt,amsfonts,array}
\usepackage{mathrsfs}
\usepackage[english]{babel}
\usepackage{latexsym}
\usepackage{amssymb}
\usepackage{euscript}
\usepackage{graphicx}
\usepackage{enumerate}
\usepackage{nicematrix}
\usepackage{url}
\usepackage{mathtools}

\newcommand\norm[1]{\lVert#1\rVert}

\newcommand\scal[2]{\left\langle #1,#2\right\rangle}

\theoremstyle{definition}

\theoremstyle{remark}

\numberwithin{equation}{section}

\newtheorem{theorem}{\bf Theorem}[section]
\newtheorem{definition}[theorem]{\bf Definition}
\newtheorem{lemma}[theorem]{\bf Lemma}
\newtheorem{corollary}[theorem]{\bf Corollary}
\newtheorem{remark}[theorem]{\bf Remark}

\newtheorem{conjecture}{\bf Conjecture}
\usepackage{enumerate}

\newcommand{\R}{\ensuremath{\mathcal{R}}}

\def\d{\begin{definition}}
	\def\ed{\end{definition}}
\def\t{\begin{theorem}}
	\def\et{\end{theorem}}

\def\B{{\mathfrak B}}

\def\H{{\cal H}}
\def\K{{\cal K}}
\def\L{{\cal L}}
\def\M{{\cal M}}
\def\T{{\cal T}}

\def\R{{\cal R}}

\def\V{{\cal V}}

\def\BK{\mathfrak{B}({\mathcal{K}})}

\def\sbmatrix{\left[\begin{array}}
	\def\endsbmatrix{\end{array}\right]}
\def\d{\begin{definition}}
	\def\ed{\end{definition}}
\def\t{\begin{theorem}}
	\def\et{\end{theorem}}
\def\c{\begin{corollary}}
	\def\ec{\end{corollary}}
\def\l{\begin{lemma}}
	\def\el{\end{lemma}}

\journal{arXiv}

\begin{document}
	
	\begin{frontmatter}
		
		
		
		\title{Spherically quasinormal tuples: $n$-th root problem and hereditary properties}
		
		
		\author{Hranislav Stankovi\' c}
		\affiliation{organization={Department of Mathematics, Faculty of Electronic Engineering},
			addressline={Aleksandra Medvedeva 14}, 
			city={Nis},
			postcode={18115},
			country={Serbia}}
		\begin{abstract}
			In this paper, we provide several characterizations of a spherically quasinormal tuple $\mathbf{T}$  in terms of its normal extension, as well as in terms of powers of the associated elementary operator $\Theta_{\mathbf{T}}(I)$. Utilizing these results, we establish that the powers of spherically quasinormal tuples remain  spherically quasinormal. Additionally, we prove that the subnormal $n$-roots of spherically quasinormal tuples must also be spherically quasinormal, thereby resolving a multivariable version of a previously posed problem by Curto et al. in \cite{CurtoLeeYoon20}. Furthermore, we investigate the connection between a (pure) spherically quasinormal tuple $\mathbf{T}$, its minimal normal extension $\mathbf{N}$, and its dual $\mathbf{S}$. Among other things, we show that $\mathbf{T}$ inherits the spherical polar decomposition from $\mathbf{N}$.  Finally, we also demonstrate that $\mathbf{N}$ is Taylor invertible if and only if $\mathbf{T}$ and $\mathbf{S}$ have closed ranges.
		\end{abstract}
		
		
		
		
	\end{frontmatter}
	
	\section{Introduction}
	Let $\mathcal{H}$ be a complex Hilbert space and let $\mathfrak{B}(\mathcal{H})$ denote the algebra of bounded linear operators on $\mathcal{H}$. With $0_\mathcal{H}$ and $I_\mathcal{H}$ we denote the zero and the identity operator, respectively, where the subscript is omitted if it does not lead to a  confusion regarding the space. Also, $\mathcal{N}(T)$ and $\mathcal{R}(T)$ shall represent the null-space and the range of an operator $T\in\mathfrak{B}(\mathcal{H})$. The adjoint of an operator $T\in\mathfrak{B}(\mathcal{H})$ will be denoted by $T^*$. A closed subspace $\L$ is a reducing subspace of an operator $T\in\B(\H)$ if it is invariant under both $T$ and $T^*$, i.e., if $T(\L)\subseteq \L$ and $T^*(\L)\subseteq \L$. For a closed subspace $\M\subseteq\H$, we write $P_\M$ to denote the orthogonal projection onto $\M$. The corestriction $A^{cr}$ of an operator $A: \H\to \K$ is defined as a map with domain $\H$, codomain $\R(A)$ and 
	\begin{equation*}
		A^{cr}x=Ax, \quad x\in \H.
	\end{equation*}
	\par 
	An operator $T$ is said to be \textit{normal} if it commutes with its adjoint, i.e., if $T^*T=TT^*$. Also,  $T$ is pure if it has no non-zero reducing subspace on which
	it is normal, i.e., if it has no normal direct summand.
	
	Due to the importance of the normal operators, their many generalizations have appeared over the decades. Definitely the most important classes of such operators are the following:
	\begin{itemize}
		\item quasinormal operators: $T$ commutes with $T^*T$, i.e., $TT^*T=T^*T^2$;
		\item subnormal operators: there exist a Hilbert space $\K$, $\K\supseteq \H$, and a normal operator $N\in \B(\K)$ such that 
		$$
		N=\begin{bmatrix} T&\ast\\0&\ast\end{bmatrix}: \begin{pmatrix}
			\mathcal{H}\\
			\mathcal{H}^\perp
		\end{pmatrix}\to \begin{pmatrix}
			\mathcal{H}\\
			\mathcal{H}^\perp
		\end{pmatrix};
		$$
		\item hyponormal operators: if $TT^*\leq T^*T$.
	\end{itemize}
	It is well known that 
	\begin{equation*}
		\text{normal}\,\Rightarrow\,\text{quasinormal}\,\Rightarrow\,\text{subnormal}\,\Rightarrow\,\text{hyponormal}.
	\end{equation*}
	For other generalizations of normal operators, see , for example, \cite{Furuta01}.
	
	A normal operator $N\in\BK$ is the minimal extension of a subnormal $T\in\B(\H)$ if $\H\subseteq \K$ and if  $\K'\subseteq \K$ is a reducing subspace of $N$ and $\H\subseteq \K'$, then $\K'=\K$. 
	Every subnormal operator has a minimal normal extension
	$N$, and $N$ is unique up to unitary equivalence (see \cite[Corollary 2.7]{Conway91}).

	\bigskip 
	
	The previous notions have been also transfered to a multivariable operator theory setting. First, for $\mathbf{T}=(T_1,\ldots,T_d)\in\mathfrak{B}(\mathcal{H})^d$, by $\mathbf{T^*}$ we denote the operator $d$-tuple $\mathbf{T^*}=(T_1^*,\ldots,T_d^*)\in\mathfrak{B}(\mathcal{H})^d$. For $S,T\in\mathfrak{B}(\mathcal{H})$ let $[S,T]=ST-TS$. We say that an $d$-tuple $\mathbf{T}=(T_1,\ldots,T_d)\in\mathfrak{B}(\mathcal{H})^d$ of operators on $\mathcal{H}$ is jointly \textit{hyponormal} if the operator matrix
	\begin{equation*}
		[\mathbf{T^*},\mathbf{T}]:=\begin{bmatrix}
			[T_1^*,T_1]&[T_2^*,T_1]&\cdots&[T_d^*,T_1]\\
			[T_1^*,T_2]&[T_2^*,T_2]&\cdots&[T_d^*,T_2]\\
			\vdots&\vdots&\ddots&\vdots\\
			[T_1^*,T_d]&[T_2^*,T_d]&\cdots&[T_d^*,T_d]
		\end{bmatrix}
	\end{equation*}
	is positive on the direct sum of $n$ copies of $\mathcal{H}$ (cf. \cite{Athavale88}, \cite{Curto90}, \cite{CurtoMuhlyXia88}). A $d$-tuple $\mathbf{T}$ is said to be \textit{normal} if $\mathbf{T}$ is commuting and each $T_k$ is normal, $k\in\{1,\ldots, d\}$,  and \textit{subnormal} if there exist a Hilbert space $\K$, $\K\supseteq \H$, and a normal $d$-tuple $\mathbf{N}=(N_1,\ldots,N_d)\in \B(\K)^d$ such that 
	$$
	N_k=\begin{bmatrix} T_k&\ast\\0&\ast\end{bmatrix}: \begin{pmatrix}
		\mathcal{H}\\
		\mathcal{H}^\perp
	\end{pmatrix}\to \begin{pmatrix}
		\mathcal{H}\\
		\mathcal{H}^\perp
	\end{pmatrix}, \quad k\in\{1,\ldots,d\}.
	$$
	A $d$-tuple $\mathbf{N}\in\B(\K)^d$ is called the minimal normal extension for a subnormal tuple $\mathbf{T}$ if there is no proper closed subspace of $\mathcal{K}$, containing $\mathcal{H}$ and which is reducing for each $N_k$, $k\in\{1,\ldots,d\}$. \par 
	An operator $d$-tuple $\mathbf{T}=(T_1,\ldots,T_d)\in\mathfrak{B}(\mathcal{H})^d$ is said to be \emph{pure} if there is no non-trivial subspace $\mathcal{M}\subseteq\mathcal{H}$, which is a common reducing subspace for $T_1,\ldots,T_d$ and for which operators $P^{cr}T_1\restriction_{\mathcal{M}},\ldots, P^{cr}T_d\restriction_{\mathcal{M}}$ form a normal $d$-tuple on $\mathcal{M}$.

	The notion of quasinormality has  several multivariable analogues. An operator $d$-tuple $\mathbf{T}=(T_1,\ldots,T_d)\in\mathfrak{B}(\mathcal{H})^d$ is called \textit{matricially quasinormal} if each $T_i$ commutes with each $T_j^*T_k$, $i,j,k\in\{1,\ldots,d\}$, $\mathbf{T}$ is  \textit{jointly quasinormal} if each $T_i$ commutes with each $T_j^*T_j$, $i,j\in\{1,\ldots,d\}$, and \textit{spherically quasinormal} if each $T_i$ commutes with $\sum_{j=1}^{n}T_j^*T_j$, $i\in\{1,\ldots,d\}$ (see \cite{Gleason06}). As shown in \cite{AthavalePodder15} and \cite{Gleason06}, we have
	
	\begin{align*}
		\text{normal}\,\Rightarrow\,&\text{matricially quasinormal}\,\Rightarrow\, \text{jointly quasinormal}\,\\\Rightarrow\,&\text{spherically quasinormal}\,\Rightarrow\,\text{subnormal}.
	\end{align*}
	For more details on these classes and the relations between them, see \cite{CurtoLeeYoon05, CurtoLeeYoon07, CurtoYoon06}.
	\bigskip

	Let $d\in\mathbb{N}$.	For $T_1,\ldots, T_d\in\mathfrak{B}(\mathcal{H})$, consider a $d$-tuple $\mathbf{T}=\begin{pmatrix}
		T_1\\\vdots\\T_d
	\end{pmatrix}$ as an operator from $\mathcal{H}$ into $\mathcal{H}\oplus\cdots\oplus\mathcal{H}$, that is,
	\begin{equation*}
		\mathbf{T}=\begin{pmatrix}
			T_1\\\vdots\\T_d
		\end{pmatrix}:\,\mathcal{H}\to \begin{array}{c}
			\mathcal{H}\\\oplus\\\vdots\\\oplus\\\mathcal{H}
		\end{array}.
	\end{equation*}
	
	We have that the kernel and range of $\mathbf{T}$ are, respectively,
	\begin{equation*}
		\mathcal{N}(\mathbf{T})=\bigcap_{k=1}^d\mathcal{N}(T_k),
	\end{equation*}
	
	and
	\begin{equation*}
		\mathcal{R}(\mathbf{T})=\left\{\begin{pmatrix}
			T_1x\\\vdots\\T_dx
		\end{pmatrix}:\, x\in\mathcal{H}\right\}.
	\end{equation*}
	
	We define (canonical) spherical polar decomposition of $\mathbf{T}$ (cf. \cite{CurtoYoon16}, \cite{CurtoYoon18}, \cite{KimYoon17}) as
	
	\begin{equation*}
		\mathbf{T}=\begin{pmatrix}
			T_1\\\vdots\\T_d
		\end{pmatrix}=\begin{pmatrix}
			V_1\\\vdots\\V_d
		\end{pmatrix}P=\begin{pmatrix}
			V_1P\\\vdots\\V_dP
		\end{pmatrix}=\mathbf{V}P,
	\end{equation*}
	where $P=\sqrt{T_1^*T_1+\cdots+T_d^*T_d}$ is a positive operator on $\mathcal{H}$, and \begin{equation*}
		\mathbf{V}=\begin{pmatrix}
			V_1\\\vdots\\V_d
		\end{pmatrix}:\,\mathcal{H}\to \begin{array}{c}
			\mathcal{H}\\\oplus\\\vdots\\\oplus\\\mathcal{H}
		\end{array},
	\end{equation*}
	is a spherical partial isometry from $\mathcal{H}$ into $\mathcal{H}\oplus\ldots\oplus\mathcal{H}$. In other words, $V_1^*V_1+\cdots+V_d^*V_d$ is the (orthogonal) projection onto the initial space of the partial isometry $\mathbf{V}$ which is 
	\begin{align*}
		\mathcal{N}(\mathbf{T})^\perp&=\left(\bigcap_{i=k}^d\mathcal{N}(T_k)\right)^\perp=\mathcal{N}(P)^\perp=\left(\bigcap_{i=k}^d\mathcal{N}(V_k)\right)^\perp.
	\end{align*}

	Operator $d$-tuple $\mathbf{T}\in\mathfrak{B}(\mathcal{H})^d$ is said to be \emph{Taylor invertible} if its associated Koszul complex $\mathcal{K}(\mathbf{T},\mathcal{H})$ is exact. For $d=2$, the Koszul complex $\mathcal{K}(\mathbf{T},\mathcal{H})$ associated to $\mathbf{T}=(T_1,T_2)$ on $\mathcal{H}$ is given by:
	\begin{equation*}
		\mathcal{K}(\mathbf{T},\mathcal{H}):\quad 0\longrightarrow\mathcal{H}\stackrel{\mathbf{T}}{\longrightarrow}\mathcal{H}\oplus\mathcal{H}\stackrel{(-T_2 \,\,T_1)}{\longrightarrow}\mathcal{H}\longrightarrow0,
	\end{equation*}
	where $\mathbf{T}=\begin{pmatrix}
		T_1\\T_2
	\end{pmatrix}$. We also say that $\mathbf{T}\in \mathfrak{B}(\mathcal{H})^d$ is \textit{left Taylor invertible} if $\mathcal{N}(\mathbf{T})=\{0_\mathcal{H}\}$ and $\mathcal{R}(\mathbf{T})$ is a closed subspace of $\mathcal{H}^d$.
	For more information on the Taylor invertibility and Koszul complexes, we refer a reader to \cite{BenhidaZerouali07, BenhidaZerouali11, KimYoon19, Taylor70a, Taylor70b}.\par 
	The Taylor spectrum   of a commuting $d$-tuple $\mathbf{T}=(T_1,\ldots,T_d) \in \mathfrak B(\mathcal H)^d$ is denoted by $\sigma_T(\mathbf T)$ and it is defined as
	\begin{equation*}
		\sigma_T(\mathbf{T})=\left\{(\lambda_1,\ldots,\lambda_d)\in\mathbb{C}^d:\,  \mathcal{K}((T_1-\lambda_1,\ldots, T_d-\lambda_d),\mathcal{H})\,\text{is not exact}    \right\}.
	\end{equation*}

	Finally, we introduce several notations which will be frequently used throughout the paper. For $\mathbf{T}=(T_1,\ldots, T_m)\in\mathfrak{B}(\mathcal{H})^m$ and $\mathbf{S}=(S_1,\ldots, S_n)\in\mathfrak{B}(\mathcal{H})^n$, we write
	\begin{equation}\label{product_def}
		\mathbf{T}\circ\mathbf{S}:=(T_1S_1,\ldots,T_1S_n,\ldots, T_mS_1,\ldots, T_mS_n).
	\end{equation}
	If $m=n$, we also define
	\begin{equation*}
		\mathbf{T}\mathbf{S}:=(T_1S_1,\ldots,T_mS_m).
	\end{equation*}
	Related to this definitions, we introduce the following notations:
	\begin{equation*}
		\mathbf{T}^1=\mathbf{T}\quad \text{ and }\quad \mathbf{T}^{n+1}=\mathbf{T}\circ\mathbf{T}^n,
	\end{equation*}
	as well as
	\begin{equation*}
		\mathbf{T}^{(1)}=\mathbf{T}\quad \text{ and }\quad \mathbf{T}^{(n+1)}=\mathbf{T}\mathbf{T}^{(n)},
	\end{equation*}
	where $n\in\mathbb{N}$. Also, for brevity, for $S\in\mathfrak{B}(\mathcal{H})$ and $\mathbf{T}=(T_1,\ldots,T_d)\in\mathfrak{B}(\mathcal{H})^d$, we write $S\mathbf{T}$ and $\mathbf{T}S$ to denote $\mathbf{S}\mathbf{T}$ and $\mathbf{T}\mathbf{S}$, respectively, where $\mathbf{S}=(S,\ldots,S)\in\mathfrak{B}(\mathcal{H})^d$.\par 
	Now, for a given $\mathbf{T}=(T_1,\ldots, T_d)\in\mathfrak{B}(\mathcal{H})^d$, we define the operator $\Theta_\mathbf{T}:\mathfrak{B}(\mathcal{H})\longrightarrow \mathfrak{B}(\mathcal{H})$ as
	\begin{equation*}
		X\longmapsto \Theta_\mathbf{T}(X):=\sum_{k=1}^d T_k^* XT_k.
	\end{equation*}
	We define inductively $\Theta_\mathbf{T}^0(X)=X$ and $\Theta_\mathbf{T}^n(X):=\Theta_\mathbf{T}[\Theta_\mathbf{T}^{n-1}(X)]$, for all $n\in\mathbb{N}$. This operator will play a crucial role in our study.
	
	\bigskip
	The paper is organized as follows. In Section \ref{power_root_sec}, we characterize a spherically quasinormal tuple $\mathbf{T}$ in terms of its normal extension, as well as in terms of powers of operator $\Theta_{\mathbf{T}}(I)$. Using the obtained characterization, we show that the powers of spherically quasinormal tuples are also spherically quasinormal. We also demonstrate that the subnormal $n$-roots of spherically quasinormal tuples must also be spherically quasinormal, thereby providing an affirmative answer to a multivariable version of \cite[Problem 2.1]{CurtoLeeYoon20}.	In Section \ref{mne_dual_rel_sec}, we explore the relations between a spherically quasinormal tuple $\mathbf{T}$ and its minimal normal extension $\mathbf{N}$. We show that $\mathbf{T}$ "inherits" the spherical polar decomposition from $\mathbf{N}$. We also consider some spectral properties of $\mathbf{N}$, and their connections with the diagonal blocks of $\mathbf{N}$.
	
	\section{Powers and roots of spherically quasinormal tuples}\label{power_root_sec}

	We begin this section by providing the characterization of spherically quasinormal tuples, which would represent the foundation for our future study. It also has a significance on its own, as it demonstrates that  the most characterizations of the quasinormal operators in one dimensional case (see \cite{Brown53, Campbell75, ConwayOlin77, Embry73, JablonskiJungStochel14}) also hold in a multivariable operator theory setting.
	
	\begin{theorem}\label{charact}
		Let $\mathbf{T} = (T_1,\ldots,T_d)\in\mathfrak{B}(\mathcal{H})^d$ be a commuting $d$-tuple. Then the following conditions are equivalent:
		\begin{enumerate}[$(i)$]
			\item $\mathbf{T}$ is spherically quasinormal;
			\item $\Theta_{\mathbf{T}}^n(I)=\left[\Theta_{\mathbf{T}}(I)\right]^n$, for all $n\in\mathbb{N}$;
			\item there exists a spectral measure $E$ on $\mathbb{R}_+$ such that
			\begin{equation}\label{measure}
				\Theta_{\mathbf{T}}^n(I)=\int_{\mathbb{R}_+}x^n\, E(dx),\quad \text{ for all }\,n\in\mathbb{N};
			\end{equation}
			\item $\Theta_{\mathbf{T}}^n(I)=\left[\Theta_{\mathbf{T}}(I)\right]^n$, for $n\in\{2,3\}$.
			\item $\mathbf{T}$ is subnormal and $\mathcal{H}$ is invariant for $\Theta_{\mathbf{N}}(I)$, where $\mathbf{N}$ is any normal extension of $\mathbf{T}$;
			\item $\mathbf{V}P=P\mathbf{V}$, where $\mathbf{T}=\mathbf{V}P$ is the spherical polar decomposition of $\mathbf{T}$.
		\end{enumerate}
	\end{theorem}

	\begin{proof}
		$(i)\Rightarrow(ii)$: Suppose that $\mathbf{T}$ is spherically quasinormal and let us show that  
		\begin{equation}\label{power_switch}
			\Theta_{\mathbf{T}}^n(I)=\left[\Theta_{\mathbf{T}}(I)\right]^n
		\end{equation} 
		holds for all $n\in\mathbb{N}$, by using the induction principle. Clearly, the statement is true for $n=1$. Assume that it is also true for some $n\in\mathbb{N}$. Then, using the fact that $T_k$ commutes with $\Theta_{\mathbf{T}}(I)$, for all $k\in\{1,\ldots,d\}$, we have
		\begin{align*}
			\Theta_{\mathbf{T}}^{n+1}(I)&=\Theta_{\mathbf{T}}(\Theta_{\mathbf{T}}^n(I))=\Theta_{\mathbf{T}}(\left[\Theta_{\mathbf{T}}(I)\right]^n)\\
			&=\sum_{k=1}^{d}T_k^*\left[\Theta_{\mathbf{T}}(I)\right]^nT_k=\left(\sum_{k=1}^{d}T_k^*T_k\right)\left[\Theta_{\mathbf{T}}(I)\right]^n\\
			&=\Theta_{\mathbf{T}}(I)\left[\Theta_{\mathbf{T}}(I)\right]^n=\left[\Theta_{\mathbf{T}}(I)\right]^{n+1}
		\end{align*}
		Hence, \eqref{power_switch} holds for all $n\in\mathbb{N}$.\par 
		$(ii)\Rightarrow(iii)$: Let $E$ be the spectral measure of positive operator $\Theta_{\mathbf{T}}(I)$. Then, for each $n\in\mathbb{N}$,
		\begin{equation*}
			\Theta_{\mathbf{T}}^n(I)=\left[\Theta_{\mathbf{T}}(I)\right]^n=\int_{\mathbb{R}_+}x^n\, E(dx).
		\end{equation*}
		$(iii)\Rightarrow(ii)$: If $E$ is a spectral measure of some positive operator $P\in\mathfrak{B}(\mathcal{H})$ such that \eqref{measure} is satisfied, then $\Theta_{\mathbf{T}}^n(I)=P^n$ for all $n\in\mathbb{N}$. Taking $n=1$ yields $P=\Theta_{\mathbf{T}}(I)$, from where $(ii)$ directly follows.\par 
		$(ii)\Rightarrow(iv)$: This is obvious.\par 
		$(iv)\Rightarrow(i)$: Assume that $\Theta_{\mathbf{T}}^2(I)=\left[\Theta_{\mathbf{T}}(I)\right]^2$ and $\Theta_{\mathbf{T}}^3(I)=\left[\Theta_{\mathbf{T}}(I)\right]^3$. Then, we have 
		\begin{align*}
			\left[\Theta_{\mathbf{T}}(I)\right]^3&=\Theta_{\mathbf{T}}^3(I)=\Theta_{\mathbf{T}}(\Theta_{\mathbf{T}}^2(I))=\Theta_{\mathbf{T}}(\left[\Theta_{\mathbf{T}}(I)\right]^2),\\
			\left[\Theta_{\mathbf{T}}(I)\right]^3&=\Theta_{\mathbf{T}}(I)\left[\Theta_{\mathbf{T}}(I)\right]^2=\Theta_{\mathbf{T}}(I)\Theta_{\mathbf{T}}^2(I),\\
			\left[\Theta_{\mathbf{T}}(I)\right]^3&=\left[\Theta_{\mathbf{T}}(I)\right]^2\Theta_{\mathbf{T}}(I)=\Theta_{\mathbf{T}}^2(I)\Theta_{\mathbf{T}}(I).
		\end{align*}

		Now let $x\in\mathcal{H}$ be arbitrary and set $N(x):=\sum_{k=1}^d\norm{\left(\Theta_{\mathbf{T}}(I)T_k-T_k\Theta_{\mathbf{T}}(I)\right)x}^2$. Then
		\begin{align*}
			N(x)&=\sum_{k=1}^d\scal{\left(\Theta_{\mathbf{T}}(I)T_k-T_k\Theta_{\mathbf{T}}(I)\right)x}{\left(\Theta_{\mathbf{T}}(I)T_k-T_k\Theta_{\mathbf{T}}(I)\right)x}\\
			&=\sum_{k=1}^d\scal{\left(T_k^*\Theta_{\mathbf{T}}(I)-\Theta_{\mathbf{T}}(I)T_k^*\right)\left(\Theta_{\mathbf{T}}(I)T_k-T_k\Theta_{\mathbf{T}}(I)\right)x}{x}\\
			&=\sum_{k=1}^d\scal{T_k^*\left[\Theta_{\mathbf{T}}(I)\right]^2T_kx}{x}-\sum_{k=1}^d\scal{T_k^*\Theta_{\mathbf{T}}(I)T_k\Theta_{\mathbf{T}}(I)x}{x}\\
			&\,\quad -\sum_{k=1}^d\scal{\Theta_{\mathbf{T}}(I)T_k^*\Theta_{\mathbf{T}}(I)T_kx}{x}+\sum_{k=1}^d\scal{\Theta_{\mathbf{T}}(I)T_k^*T_k\Theta_{\mathbf{T}}(I)x}{x}\\
			&=\scal{\sum_{k=1}^dT_k^*\left[\Theta_{\mathbf{T}}(I)\right]^2T_kx}{x}-\scal{\left(\sum_{k=1}^dT_k^*\Theta_{\mathbf{T}}(I)T_k\right)\Theta_{\mathbf{T}}(I)x}{x}\\
			&\,\quad -\scal{\Theta_{\mathbf{T}}(I)\left(\sum_{k=1}^dT_k^*\Theta_{\mathbf{T}}(I)T_k\right)x}{x}+\scal{\Theta_{\mathbf{T}}(I)\left(\sum_{k=1}^dT_k^*T_k\right)\Theta_{\mathbf{T}}(I)x}{x}\\
			&=\scal{\Theta_{\mathbf{T}}(\left[\Theta_{\mathbf{T}}(I)\right]^2)x}{x}-\scal{\Theta_{\mathbf{T}}^2(I)\Theta_{\mathbf{T}}(I)x}{x}-\scal{\Theta_{\mathbf{T}}(I)\Theta_{\mathbf{T}}^2(I)x}{x}+\scal{\left[\Theta_{\mathbf{T}}(I)\right]^3x}{x}\\
			&=0.
		\end{align*}
		Therefore, $N(x)=0$ for each $x\in\mathcal{H}$, from where it follows that $$\norm{ \Theta_{\mathbf{T}}(I)T_k-T_k\Theta_{\mathbf{T}}(I) }=0$$ for each $k\in\{1,\ldots,d\}$. Hence, $T_k$ commutes with $\Theta_{\mathbf{T}}(I)$ for each $k\in\{1,\ldots,d\}$, i.e., $\mathbf{T}$ is spherically quasinormal, as desired.\par 
		$(i)\Leftrightarrow(v)$: This follows from \cite[Lemma 2.5]{Stankovic23a}.\par 
		$(i)\Leftrightarrow(vi)$: This was proved in \cite[Lemma 2.1]{CurtoYoon19}.
	\end{proof}

	\begin{remark}
		Another characterizations of the spherically quasinormal tuples in terms of the operator transforms can be found in \cite{CurtoYoon19}, \cite{Stankovic24} and \cite{Stankovic24a}.
	\end{remark}

	\begin{remark}\label{remark_theta_N}
		Let $\mathbf{T}=(T_1,\ldots,T_d)\in\mathfrak{B}(\mathcal{H})^d$ be a  spherically quasinormal operator tuple and let  $\mathbf{N}=(N_1,\ldots,N_d)\in\mathfrak{B}(\mathcal{K})^d$, $\mathcal{K}\supseteq\mathcal{H}$, be a normal extension $\mathbf{T}$ given by
		\begin{equation}\label{dual}
			N_k=\begin{bmatrix}
				T_k&A_k\\
				0&S_k^*
			\end{bmatrix}:\,\begin{pmatrix}
				\mathcal{H}\\\mathcal{H}^\perp
			\end{pmatrix}\to \begin{pmatrix}
				\mathcal{H}\\\mathcal{H}^\perp
			\end{pmatrix},\quad k\in\{1,\ldots,d\},
		\end{equation}
		for some $\mathbf{S}=(S_1,\ldots,S_d)\in\mathfrak{B}(\mathcal{H^\perp})^d$, where $\H^\perp=\K\ominus\H$.
		Then
		\begin{equation}\label{theta_N_representation}
			\Theta_{\mathbf{N}}(I)=\begin{bmatrix}
				\Theta_{\mathbf{T}}(I)&0\\
				0&\Theta_{\mathbf{S}}(I)
			\end{bmatrix}:\,\begin{pmatrix}
				\mathcal{H}\\\mathcal{H}^\perp
			\end{pmatrix}\to \begin{pmatrix}
				\mathcal{H}\\\mathcal{H}^\perp
			\end{pmatrix}.
		\end{equation} 
		Indeed, using the normality of $\mathbf{T}$, by direct computation, we have
		\begin{align*}
			\Theta_{\mathbf{N}}(I)&=\sum_{k=1}^dN_k^*N_k=\sum_{k=1}^d\begin{bmatrix}
				T_k^*T_k&T_k^*A_k\\
				A_k^*T_k&A_k^*A_k+S_kS_k^*
			\end{bmatrix}\\
			&=\sum_{k=1}^d\begin{bmatrix}
				T_k^*T_k&T_k^*A_k\\
				A_k^*T_k&S_k^*S_k
			\end{bmatrix}\\
			&=\begin{bmatrix}
				\Theta_{\mathbf{T}}(I)&\sum_{k=1}^dT_k^*A_k\\
				\sum_{k=1}^dA_k^*T_k&\Theta_{\mathbf{S}}(I)
			\end{bmatrix}.
		\end{align*}
		Now \cite[Theorem 2.8]{CurtoLeeYoon20} yields the desired conclusion.
	\end{remark}
	
	Our next goal is to demonstrate that the powers of spherically quasinormal tuples maintain their spherically quasinormal nature. To achieve this, we require certain auxiliary findings. While the following results are probably well-known and easy to verify, we have found no explicit mention of them elsewhere. Hence, for the sake of completeness, we provide them along with their proofs.
	
	\begin{lemma}\label{theta_switch_2}
		Let $\mathbf{T}=(T_1,\ldots, T_m)\in\mathfrak{B}(\mathcal{H})^m$ and $\mathbf{S}=(S_1,\ldots, S_n)\in\mathfrak{B}(\mathcal{H})^n$. Then
		\begin{equation*}
			\Theta_{\mathbf{T}\circ\mathbf{S}}=\Theta_{\mathbf{S}}\Theta_{\mathbf{T}}.
		\end{equation*}
	\end{lemma}

	\begin{proof}
		Let $X\in\mathfrak{B}(\mathcal{H})$ be arbitrary. Using \eqref{product_def}, we have
		\begin{align*}
			\Theta_{\mathbf{T}\circ\mathbf{S}}(X)&=\sum_{j=1}^n\sum_{i=1}^m(T_iS_j)^*XT_iS_j=\sum_{j=1}^n\sum_{i=1}^mS_j^*T_i^*XT_iS_j\\
			&=\sum_{j=1}^nS_j^*\left(\sum_{i=1}^mT_i^*XT_i\right)S_j=\Theta_{\mathbf{S}}(\Theta_{\mathbf{T}}(X))\\
			&=\Theta_{\mathbf{S}}\Theta_{\mathbf{T}}(X),
		\end{align*}
		from where the conclusion follows.
	\end{proof}
	
	Directly from the previous lemma and the induction principle, we may also conclude that the following holds:
	\begin{corollary}\label{theta_switch_n}
		Let $n\in\mathbb{N}$ and $\mathbf{T}_1,\ldots,\mathbf{T}_n$ be arbitrary operator tuples (not necessarily of the same length). Then
		\begin{equation*}
			\Theta_{\mathbf{T}_1\circ\ldots\circ\mathbf{T}_n}=\Theta_{\mathbf{T}_n}\cdots\,\Theta_{\mathbf{T}_1}.
		\end{equation*}
	\end{corollary}

	\begin{corollary}\label{theta_powers}
		Let $\mathbf{T} = (T_1,\ldots,T_d)\in\mathfrak{B}(\mathcal{H})^d$. Then
		\begin{equation*}
			\Theta_{\mathbf{T}^n}^k=\Theta_{\mathbf{T}}^{nk}
		\end{equation*}
		for all $n,k\in\mathbb{N}$.
	\end{corollary}
	\begin{proof}
		Using Corollary \ref{theta_switch_n}, we have that for all $n,k\in\mathbb{N}$,
		\begin{equation*}
			\Theta_{\mathbf{T}^n}^k=\underbrace{\Theta_{\mathbf{T}^n}\cdots\, \Theta_{\mathbf{T}^n}}_{k-\text{times}}=\Theta_{\scriptsize \underbrace{\mathbf{T}^n\circ\cdots\circ \mathbf{T}^n}_{k-\text{times}}}=\Theta_{\mathbf{T}^{nk}}=\underbrace{\Theta_{\mathbf{T}}\cdots\, \Theta_{\mathbf{T}}}_{nk-\text{times}}=\Theta_{\mathbf{T}}^{nk}.
		\end{equation*}
	\end{proof}

	We are now ready to prove the claim regarding powers of spherically quasinormal tuples.
	\begin{theorem}\label{sphquasipowers}
		Let $\mathbf{T} = (T_1,\ldots,T_d)\in\mathfrak{B}(\mathcal{H})^d$ be a spherically quasinormal $d$-tuple. Then $\mathbf{T}^n$ is  spherically quasinormal for all $n\in\mathbb{N}$.
	\end{theorem}
	
	\begin{proof}
		Assume that $\mathbf{T}$ is spherically quasinormal  and let $n, k\in\mathbb{N}$ be arbitrary. By Theorem \ref{charact} and Corollary \ref{theta_powers}, we have that
		\begin{equation*}
			\Theta_{\mathbf{T}^n}^k(I)=\Theta_{\mathbf{T}}^{nk}(I)=\left[\Theta_{\mathbf{T}}(I)\right]^{nk}=\left(\left[\Theta_{\mathbf{T}}(I)\right]^{n}\right)^k=\left[\Theta_{\mathbf{T}}^n(I)\right]^k=\left[\Theta_{\mathbf{T}^n}(I)\right]^k.
		\end{equation*}
		By invoking Theorem \ref{charact} once again, we conclude that $\mathbf{T}^n$ is spherically quasinormal.
	\end{proof}
	
	The converse of the previous theorem does not hold in general, even in a one-dimensional case. For example, $T=\begin{bsmallmatrix}
		0&1\\ 0&0
	\end{bsmallmatrix}$ is not quasinormal (normal), while $T^2$ certainly is. Thus, Curto et al. in \cite{CurtoLeeYoon20} raised the following question: "Let T be a subnormal operator, and
	assume that $T^2$ is quasinormal. Does it follow that T is quasinormal?" In \cite[Theorem 1.2]{PietrzyckiStochel21} and \cite[Theorem 2.2]{Stankovic23a} it was shown that, under the assumption that $T$ is subnormal, and $T^n$ is quasinormal, for some $n\in\mathbb{N}$, then $T$ must also be quasinormal. The next theorem shows that the analogous statement is also true in more than one dimension, i.e., we answer a multivariable version of the mentioned question in the affirmative, by considering subnormal and spherically quasinormal tuples.
	
	\begin{theorem}\label{spquasiroots}
		Let $\mathbf{T} = (T_1,\ldots,T_d)\in\mathfrak{B}(\mathcal{H})^d$ be a subnormal $d$-tuple. If $\mathbf{T}^n$ is  spherically quasinormal for some $n\in\mathbb{N}$, then $\mathbf{T}$ is also spherically quasinormal.
	\end{theorem}
	
	\begin{proof}
		Let $\mathbf{N} = (N_1,\ldots,N_d)\in\mathfrak{B}(\mathcal{K})^d$ be an arbitrary normal extension of $\mathbf{T}$, where $\mathcal{K}\supseteq\mathcal{H}$. Using the commutativity of $\mathbf{N}$ and Fuglede-Putnam theorem (\cite{Fuglede50}), it is easy to see that $\mathbf{N}^n$ is a normal extension of $\mathbf{T}^n$. Since $\mathbf{T}^n$ is, by assumption, spherically quasinormal, Theorem \ref{charact} yields that $\mathcal{H}$ is invariant for $\Theta_{\mathbf{N}^n}(I)$. By Corollary \ref{theta_powers}, $\Theta_{\mathbf{N}^n}(I)=\Theta_{\mathbf{N}}^n(I)$, while the normality of $\mathbf{N}$ and Theorem \ref{charact} imply that $\Theta_{\mathbf{N}}^n(I)=\left[\Theta_{\mathbf{N}}(I)\right]^n$. Thus, $\mathcal{H}$ is invariant for $\left[\Theta_{\mathbf{N}}(I)\right]^n$.\\
		Now let $P_{\mathcal{H}}\in\mathfrak{B}(\mathcal{K})$ be the orthogonal projection onto $\mathcal{H}$. Then it is easy to see that $$P_{\mathcal{H}}\left[\Theta_{\mathbf{N}}(I)\right]^nP_{\mathcal{H}}=\left[\Theta_{\mathbf{N}}(I)\right]^nP_{\mathcal{H}},$$ and since $\Theta_{\mathbf{N}}(I)$ is positive, by taking adjoints, we have that $\left[\Theta_{\mathbf{N}}(I)\right]^n$ and $P_{\mathcal{H}}$ commute. Using the fact that the commutants of a positive operator and its $n$-th root coincide (see \cite[Theorem 12.23]{Rudin91} or \cite[Theorem 7.20]{Weidmann80}), we have that $\Theta_{\mathbf{N}}(I)$ and $P_{\mathcal{H}}$ commute, which implies that $\mathcal{H}$ is invariant for $\Theta_{\mathbf{N}}(I)$. Theorem \ref{charact} finally yields that $\mathbf{T}$ is spherically quasinormal.
	\end{proof}
	
	\begin{corollary}\label{nth_root_pure}
		Let $\mathbf{T} = (T_1,\ldots,T_d)\in\mathfrak{B}(\mathcal{H})^d$ be a subnormal operator such that $\mathbf{T}^n$ is pure spherically  quasinormal for some $n\in\mathbb{N}$. Then $\mathbf{T}$ is pure spherically quasinormal.
	\end{corollary} 
	
	\begin{proof}
		Quasinormality follows from Theorem \ref{spquasiroots}. If $\mathbf{T}$ is not pure, then there is a non-zero reducing subspace $\mathcal{M}$ of $\mathcal{H}$ such that $$P^{cr}_{\M}\mathbf{T}\restriction_{\mathcal{M}}=(P^{cr}_{\M}T_1\restriction_{\mathcal{M}},\ldots, P^{cr}_{\M}T_d\restriction_{\mathcal{M}})$$ is normal. By representing $\mathbf{T}$ as 
		\begin{equation*}
			T_k=\begin{bmatrix}
				T_{k,n}&0\\
				0&T_{k,p}
			\end{bmatrix}:\,\begin{pmatrix}
				\mathcal{M}\\\mathcal{M}^\perp
			\end{pmatrix}\to \begin{pmatrix}
				\mathcal{M}\\\mathcal{M}^\perp
			\end{pmatrix},\quad k\in\{1,\ldots,d\},
		\end{equation*}
		where $T_{k,n}=P^{cr}_{\M}T_k\restriction_{\mathcal{M}}$, $k\in\{1,\ldots,d\}$, it can be easily seen that $P^{cr}_{\M}\mathbf{T}^n\restriction_\mathcal{M}=(P^{cr}_{\M}\mathbf{T}\restriction_\mathcal{M})^n$ is also normal, from where it follows that $\mathbf{T}^n$ is not pure. The obtained contradiction shows that $\mathbf{T}$ must be pure.
	\end{proof}

	As a direct consequence of Theorem \ref{sphquasipowers} and Theorem \ref{spquasiroots}, we have the following:
	\begin{corollary}\label{sphpowerrootseq}
		Let $\mathbf{T}\in\mathfrak{B}(\mathcal{H})^d$ be a subnormal $d$-tuple. Then the following conditions are equivalent:
		\begin{enumerate}[$(i)$]
			\item $\mathbf{T}$ is  spherically quasinormal;
			\item $\mathbf{T}^n$ is spherically quasinormal for all $n\in\mathbb{N}$;
			\item $\mathbf{T}^n$ is spherically quasinormal for some $n\in\mathbb{N}$.
		\end{enumerate}
	\end{corollary}
	
	The next result represents an improvement of \cite[Theorem 2.8]{Stankovic23a}.
	\begin{corollary}
		Let $\mathbf{T}\in\mathfrak{B}(\mathcal{H})^d$ be a subnormal $d$-tuple such that $T_iT_j=0$, for $i,j\in\{1,\ldots,d\}$, $i\neq j$. Then the following conditions are equivalent:
		\begin{enumerate}[$(i)$]
			\item $\mathbf{T}$ is  spherically quasinormal;
			\item $\mathbf{T}^{(n)}$ is spherically quasinormal for all $n\in\mathbb{N}$;
			\item $\mathbf{T}^{(n)}$ is spherically quasinormal for some $n\in\mathbb{N}$.
		\end{enumerate}
	\end{corollary}
	\begin{proof}
		The proof follows from the fact that under the assumption $T_iT_j=0$, for $i,j\in\{1,\ldots,d\}$, $i\neq j$, we have that $\mathbf{T}^{(n)}$ is spherically quasinormal if and only if $\mathbf{T}^n$ is spherically quasinormal. Now, it only remains to apply Corollary \ref{sphpowerrootseq}.
	\end{proof}

	We end this section with the following conjecture which was recently shown  to be true in one-dimensional case (see \cite[Theorem 4.1]{PietrzyckiStochel23}), and if answered affirmatively, would extend and generalize Theorem \ref{spquasiroots}.
	\begin{conjecture}
		Let $\mathbf{T}\in\mathfrak{B}(\mathcal{H})^d$ be a jointly hyponormal $d$-tuple of commuting operators. If $\mathbf{T}^n$ is  spherically quasinormal for some $n\in\mathbb{N}$, then  $\mathbf{T}$ is also spherically quasinormal.
	\end{conjecture}
	
	\section{Relations with the minimal normal extension and the subnormal dual}\label{mne_dual_rel_sec}
	
	We begin this section by proving the hereditary property of the polar decomposition between a pure spherically quasinormal tuple, and its minimal normal extension. First, we need the following theorem, which is of independent interest.

	\begin{theorem}\label{injective_N}
		Let $\mathbf{T}\in\mathfrak{B}(\mathcal{H})^d$ be a spherically quasinormal $d$-tuple with the minimal normal extension $\mathbf{N}\in\mathfrak{B}(\mathcal{K})^d$, $\mathcal{K}\supseteq\mathcal{H}$. If $\mathbf{T}$ is pure, then $\mathbf{N}$ is injective.
	\end{theorem}
	
	\begin{proof}
		Assume that $\mathbf{T}=(T_1,\ldots,T_d)\in\mathfrak{B}(\mathcal{H})^d$ is a spherically quasinormal $d$-tuple, and let $\mathbf{T}=\mathbf{V}P=(V_1P,\ldots,V_dP)$ be the spherical polar decomposition of $\mathbf{T}$. By Theorem \ref{charact}, we have that $\mathbf{V}P=P\mathbf{V}$. Since $\mathcal{N}(\mathbf{T})=\mathcal{N}(P)$,  if $h\in\mathcal{N}(\mathbf{T})=\displaystyle\cap_{k=1}^d\mathcal{N}(T_k)$, then
		\begin{equation*}
			T_k^*h=PV_k^*h=V_k^*Ph=0.
		\end{equation*}
		Thus, $\mathcal{N}(\mathbf{T})\subseteq\displaystyle\cap_{k=1}^d\mathcal{N}(T_k^*)=\mathcal{N}(\mathbf{T^*})$. This yields that $\mathcal{N}(\mathbf{T})$ is a reducing subspace for $\mathbf{T}$. Since $\mathbf{T}$ is assumed to be pure, it must be $\mathcal{N}(\mathbf{T})=\{0\}$. This further implies that $P^2=\Theta_{\mathbf{T}}(I)$ is injective with dense range, and that $\mathbf{V}$ is a spherical isometry.\par 
		
		Now let $\mathbf{N}=(N_1,\ldots,N_d)\in\mathfrak{B}(\mathcal{K})^d$ be the minimal normal extension of $\mathbf{T}$, where $\mathcal{K}=\mathcal{H}\oplus\mathcal{H}^\perp$. By Remark \ref{remark_theta_N}, we have that $\Theta_{\mathbf{N}}(I)$ has representation given by \eqref{theta_N_representation}.
		Let $\begin{pmatrix}
			f\\g
		\end{pmatrix}\in\mathcal{N}(\mathbf{N})$ and $h\in\mathcal{H}$ be arbitrary. Then,
		\begin{align*}
			\scal{\begin{pmatrix}
					f\\g
			\end{pmatrix}}{\begin{pmatrix}
					\Theta_{\mathbf{T}}(I)h\\0
			\end{pmatrix}}
			&=\scal{\begin{pmatrix}
					f\\g
			\end{pmatrix}}{\Theta_{\mathbf{N}}(I)\begin{pmatrix}h\\0
			\end{pmatrix}}\\
			&=\scal{\begin{pmatrix}
					f\\g
			\end{pmatrix}}{\sum_{k=1}^dN_k^*N_k\begin{pmatrix}h\\0
			\end{pmatrix}}\\
			&=\sum_{k=1}^d\scal{N_k\begin{pmatrix}
					f\\g
			\end{pmatrix}}{N_k\begin{pmatrix}h\\0
			\end{pmatrix}}\\
			&=0.
		\end{align*}
		Since $\Theta_{\mathbf{T}}(I)$ has dense range, it follows that $\mathcal{N}(\mathbf{N})\perp\mathcal{H}$. In other words, $\mathcal{H}\subseteq\mathcal{N}(\mathbf{N})^\perp$ and $\mathcal{N}(\mathbf{N})^\perp$ reduces $\mathbf{N}$, as $\mathbf{N}$ is a normal $d$-tuple. Finally, by utilizing the fact that $\mathbf{N}$ is the minimal normal extension for $\mathbf{T}$, we conclude that it must be $\mathcal{N}(\mathbf{N})=\{0_{\mathcal{K}}\}$.
	\end{proof}

	\begin{theorem}
		Let $\mathbf{T}=\mathbf{V}P=(V_1P,\ldots,V_dP)$ be the spherical polar decomposition of a pure spherically quasinormal $d$-tuple $\mathbf{T}\in\mathfrak{B}(\mathcal{H})^d$ and let $\mathbf{N}\in\mathfrak{B}(\mathcal{K})^d$, $\mathcal{K}\supseteq\mathcal{H}$, be the minimal normal extension of $\mathbf{T}$. If $\mathbf{N}=\mathbf{W}R=(W_1R,\ldots,W_dR)$ is the spherical polar decomposition of $\mathbf{N}$, then $\mathbf{V}=P_{\mathcal{H}}^{cr}\mathbf{W}\restriction_{\mathcal{H}}$ and $P=P_{\mathcal{H}}^{cr}R\restriction_{\mathcal{H}}$.
	\end{theorem}

	\begin{proof}
		Assume that $\mathbf{T}=(T_1,\ldots,T_d)\in\mathfrak{B}(\mathcal{H})^d$ is a pure spherically quasinormal $d$-tuple with the minimal normal extension $\mathbf{N}=(N_1,\ldots,N_d)\in\mathfrak{B}(\mathcal{K})^d$.  By Remark \ref{remark_theta_N}, we have
		\begin{equation*}
			R=\left[\Theta_{\mathbf{N}}(I)\right]^\frac{1}{2}=\begin{bmatrix}
				\left[\Theta_{\mathbf{T}}(I)\right]^\frac{1}{2}&0\\
				0&\left[\Theta_{\mathbf{S}}(I)\right]^\frac{1}{2}
			\end{bmatrix}=\begin{bmatrix}
				P&0\\
				0&\left[\Theta_{\mathbf{S}}(I)\right]^\frac{1}{2}
			\end{bmatrix},
		\end{equation*}
		and thus, $P=P_{\mathcal{H}}^{cr}R\restriction_{\mathcal{H}}$.\par 
		Now let 
		\begin{equation*}
			W_k=\begin{bmatrix}
				W_{k,1}&W_{k,2}\\
				W_{k,3}&W_{k,4}
			\end{bmatrix}:\,\begin{pmatrix}
				\mathcal{H}\\\mathcal{H}^\perp
			\end{pmatrix}\to \begin{pmatrix}
				\mathcal{H}\\\mathcal{H}^\perp
			\end{pmatrix},\quad k\in\{1,\ldots,d\},
		\end{equation*}
		be the representation of $\mathbf{W}$ with respect to $\mathcal{K}=\mathcal{H}\oplus\mathcal{H}^\perp$, and let $k\in\{1,\ldots,d\}$ be arbitrary. Then, 
		\begin{align*}
			N_k&=\begin{bmatrix}
				W_{k,1}&W_{k,2}\\
				W_{k,3}&W_{k,4}
			\end{bmatrix} \begin{bmatrix}
				P&0\\
				0&\left[\Theta_{\mathbf{S}}(I)\right]^\frac{1}{2}
			\end{bmatrix}=\begin{bmatrix}
				W_{k,1}P&W_{k,2}X^\frac{1}{2}\\
				W_{k,3}P&W_{k,4}X^\frac{1}{2}
			\end{bmatrix}.
		\end{align*}
		Since $\mathcal{H}$ is invariant for $\mathbf{N}$, and $T_k=P_{\mathcal{H}}^{cr}N_k\restriction_{\mathcal{H}}$, we have that $T_k=W_{k,1}P$ and $W_{k,3}P=0$. As in the proof of Theorem \ref{injective_N}, we can conclude that $P$ has dense range, which further implies that $W_{k,3}=0$.\par 
		Furthermore, Theorem \ref{injective_N} also implies that $\mathcal{N}(\mathbf{W})=\mathcal{N}(\mathbf{V})=\{0_{\mathcal{K}}\}$, and thus $\mathbf{W}$ is a spherical isometry, i.e., $\Theta_{\mathbf{W}}(I)=I_{\mathcal{K}}$. This, together with
		\begin{align*}
			\Theta_{\mathbf{W}}(I)&=\sum_{k=1}^d\begin{bmatrix}
				W_{k,1}^*&0\\
				W_{k,2}^*&W_{k,4}^*
			\end{bmatrix}\begin{bmatrix}
				W_{k,1}&W_{k,2}\\
				0&W_{k,4}
			\end{bmatrix}=\sum_{k=1}^d\begin{bmatrix}
				W_{k,1}^*W_{k,1}&\ast\\
				\ast &\ast
			\end{bmatrix},
		\end{align*}
		yields that $\sum_{k=1}^d W_{k,1}^*W_{k,1}=I_{\mathcal{H}}$. In other words, $\mathbf{W}_1:=(W_{1,1},\ldots,W_{d,1})\in\mathfrak{B}(\mathcal{H})^d$ is a spherical isometry such that $\mathbf{T}=\mathbf{W}_1P$. This proves that $\mathbf{V}=P_{\mathcal{H}}^{cr}\mathbf{W}\restriction_{\mathcal{H}}$.
	\end{proof}

	Before we further proceed, recall the following definition from \cite{Athavale89}.
	
	\begin{definition}\cite[Definition 6]{Athavale89}
		Let $\mathbf{T}=(T_1,\ldots,T_d)\in\mathfrak{B}(\mathcal{H})^d$ be a (pure) subnormal $d$-tuple and let $\mathbf{N}=(N_1,\ldots,N_d)\in\mathfrak{B}(\mathcal{K})^d$, $\mathcal{K}\supseteq\mathcal{H}$, be the minimal normal extension of $\mathbf{T}$ given by \eqref{dual}.
		Then $\mathbf{S}=(S_1,\ldots,S_d)\in\mathfrak{B}(\K\ominus\H)^d$ is called the dual of $\mathbf{T}$.
	\end{definition}
	For more details on the duals of subnormal operators, especially in a one-dimensional case, see \cite{Conway81, CvetkovicIlicStankovic24, Murphy82, Saito92}.
	
	\begin{remark}\label{remark_dual}
		Assume that $\mathbf{T}=(T_1,\ldots,T_d)\in\mathfrak{B}(\mathcal{H})^d$ is a subnormal $d$-tuple with the minimal normal extension $\mathbf{N}=(N_1,\ldots,N_d)\in\mathfrak{B}(\mathcal{K})^d$, $\mathcal{K}\supseteq\mathcal{H}$, given by \eqref{dual}. Then $\mathbf{S}=(S_1,\ldots,S_d)\in\mathfrak{B}(\K\ominus\H)^d$ must be pure. Assume to the contrary, that there exists Hilbert spaces $\mathcal{L},\mathcal{M}\subseteq\mathcal{H}^\perp=\K\ominus\H$ such that  $\mathcal{M}$ is non-trivial, $\mathcal{H}^\perp=\mathcal{L}\oplus\mathcal{M}$ and
		\begin{equation*}
			S_k=\begin{bmatrix}
				S_{k,p}&0\\
				0&S_{k,n}
			\end{bmatrix}:\,\begin{pmatrix}
				\mathcal{L}\\\mathcal{M}
			\end{pmatrix}\to \begin{pmatrix}
				\mathcal{L}\\\mathcal{M}
			\end{pmatrix},\quad k\in\{1,\ldots,d\},
		\end{equation*}
		where $S_{k,n}$ is normal for all $k\in\{1,\ldots,d\}$. Then \eqref{dual} can be written as
		\begin{equation*}
			N_k=\begin{bmatrix}
				T_k&A_{k,1}&A_{k,2}\\
				0&S_{k,p}^*&0\\
				0&0&S_{k,n}^*
			\end{bmatrix}:\,\begin{pmatrix}
				\mathcal{H}\\	\mathcal{L}\\\mathcal{M}
			\end{pmatrix}\to \begin{pmatrix}
				\mathcal{H}\\	\mathcal{L}\\\mathcal{M}
			\end{pmatrix},\quad k\in\{1,\ldots,d\}.
		\end{equation*}
		Using the normality of $N_k$ and $S_{n,k}$, a direct computation shows that $A_{k,2}=0$ for all $ k\in\{1,\ldots,d\}$. This further implies that $N_k$ is normal on $\mathcal{H}\oplus\mathcal{L}\subsetneqq\mathcal{K}$ for all $ k\in\{1,\ldots,d\}$, which is a contradiction with the fact that $\mathbf{N}$ is the minimal normal extension of $\mathbf{T}$. Thus, $\mathbf{S}$ must be pure. \par 
		Under the additional assumption that $\mathbf{T}$ is pure, by using the similar technique, and representing \eqref{dual} as
		\begin{equation*}
			N_k^*=\begin{bmatrix}
				S_k&A_k^*\\
				0&T_k^*
			\end{bmatrix}:\,\begin{pmatrix}
				\mathcal{H}^\perp\\\mathcal{H}
			\end{pmatrix}\to \begin{pmatrix}
				\mathcal{H}^\perp\\\mathcal{H}
			\end{pmatrix},\quad k\in\{1,\ldots,d\},
		\end{equation*}
		we can show that $\mathbf{N^*}=(N_1^*,\ldots,N_d^*)$ is the minimal normal extension for $\mathbf{S}$.
	\end{remark}

	The next theorem first appears as \cite[Proposition 2.3]{AthavalePodder15}. Here, we provide a simplified proof of it based on the methods developed.
	
	\begin{theorem}
		Let $\mathbf{T}\in\mathfrak{B}(\mathcal{H})^d$ be a pure spherically quasinormal $d$-tuple. Then its dual is also a pure spherically quasinormal $d$-tuple.
	\end{theorem}
	
	\begin{proof}
		Let $\mathbf{N}\in\mathfrak{B}(\mathcal{K})^d$, $\mathcal{K}\supseteq\mathcal{H}$, be the minimal normal extension for $\mathbf{T}$. By Theorem \ref{charact}, $\mathcal{H}$ is an invariant subspace for $\Theta_{\mathbf{N}}(I)$. Since $\Theta_{\mathbf{N}}(I)$ is positive and $\mathbf{N}$ is normal, this is equivalent with the fact that  $\mathcal{H}^\perp$ is an invariant subspace for $\Theta_{\mathbf{N^*}}(I)$. The result now follows by Remark \ref{remark_dual} and Theorem \ref{charact}.
	\end{proof}

	We end this section by showing several spectral relations between a pure spherically quasinormal tuple, its minimal normal extension, as well as its dual.
	
	\begin{theorem}
		Let $\mathbf{T}\in\mathfrak{B}(\mathcal{H})^d$ be a pure spherically quasinormal $d$-tuple with the minimal normal extension $\mathbf{N}\in\mathfrak{B}(\mathcal{K})^d$, $\mathcal{K}\supseteq\H$, and the corresponding dual $\mathbf{S}\in\mathfrak{B}(\K\ominus\H)^d$. Then the following conditions are equivalent:
		\begin{enumerate}[$(i)$]
			\item $\mathbf{N}$ is Taylor invertible;
			\item $\mathcal{R}(\mathbf{N})$ is closed;
			\item $\mathcal{R}(\mathbf{T})$ and $\mathcal{R}(\mathbf{S})$ are closed;
			\item $\mathbf{T}$ and $\mathbf{S}$ are left Taylor invertible.
		\end{enumerate}
	\end{theorem}
	
	\begin{proof}
		$(i)\Rightarrow(ii)$: This is obvious. \par 
		$(ii)\Rightarrow(iii)$: We utilize the fact that for any operator $X$, $\mathcal{R}(X)$ is closed if and only if $\mathcal{R}(X^*X)$ is closed. Thus, by treating $\mathbf{N}$ as a column vector, and under the assumption that $\mathcal{R}(\mathbf{N})$ is closed, we have that $\Theta_{\mathbf{N}}(I)$ has closed range. Using \eqref{theta_N_representation}, it follows that $\Theta_{\mathbf{T}}(I)$ and $\Theta_{\mathbf{S}}(I)$ are closed-range operators, i.e., $\mathcal{R}(\mathbf{T})$ and $\mathcal{R}(\mathbf{S})$ are closed.\par 
		$(iii)\Rightarrow(iv)$: Since $\mathbf{T}$ is pure, by the proof of Theorem \ref{injective_N}, it must be $\mathcal{N}(\mathbf{T})=\{0_{\mathcal{H}}\}$, which together with the closedness of  $\mathcal{R}(\mathbf{T})$, imply that $\mathbf{T}$ is left Taylor invertible. By Remark \ref{remark_dual}, we have that $\mathbf{S}$ is also pure, so the similar arguments yield the left Taylor invertibility of $\mathbf{S}$.\par 
		$(iv)\Rightarrow(i)$: Now assume that $\mathbf{T}$ and $\mathbf{S}$ are left Taylor invertible. Since $\mathcal{N}(\Theta_{\mathbf{T}}(I))=\mathcal{N}(\mathbf{T})$ and $\mathcal{R}(\Theta_{\mathbf{T}}(I))$ is closed if and only if $\mathcal{R}(\mathbf{T})$ is closed, we have that $\Theta_{\mathbf{T}}(I)$ is left invertible, which is, due to the positivity of $\Theta_{\mathbf{T}}(I)$, equivalent with the invertibility of $\Theta_{\mathbf{T}}(I)$. Similarly, $\Theta_{\mathbf{S}}(I)$ is an invertible operator. Now, \eqref{theta_N_representation} implies that $\Theta_{\mathbf{N}}(I)$ is invertible, and thus $\mathbf{N}$ is Taylor invertible by  \cite[Corollary 3.4]{Curto80} (or \cite[Corollary 3.9]{Curto81b}).
	\end{proof}

	The previous theorem can also be treated a multivariable analogue of \cite[Theorem 3.6]{CvetkovicIlicStankovic24}.
	\begin{theorem}
		Let $\mathbf{T}\in\mathfrak{B}(\mathcal{H})^d$ be a pure subnormal $d$-tuple with the minimal normal extension $\mathbf{N}\in\mathfrak{B}(\mathcal{K})^d$, $\mathcal{K}\supseteq\H$, and the corresponding dual $\mathbf{S}\in\mathfrak{B}(\K\ominus\H)^d$. Then $\mathbf{T}$ is Taylor invertible if and only if $\mathbf{S}$ is Taylor invertible.
	\end{theorem}
	
	\begin{proof}
		Assume that $\mathbf{T}$ is Taylor invertible. By \cite{Putinar84} (cf. \cite[Theorem 1]{Curto81a}), we have that 
		\begin{equation*}
			\sigma_T(\mathbf{N})\subseteq\sigma_T(\mathbf{T}),
		\end{equation*}
		and thus, $\mathbf{N}$ is Taylor invertible. The Taylor invertibility of $\mathbf{S}$ now follows from \cite[Lemma 4.3]{Curto80} and \cite[Lemma 4.4]{Curto80}. \par 
		The converse can be done in a similar manner, having in mind Remark \ref{remark_dual}.
	\end{proof}
	
	Since every spherically quasinormal tuple is subnormal (see \cite[Proposition 2.1]{AthavalePodder15}), we have the following:
	\begin{corollary}
		Let $\mathbf{T}\in\mathfrak{B}(\mathcal{H})^d$ be a pure spherically quasinormal $d$-tuple with the minimal normal extension $\mathbf{N}\in\mathfrak{B}(\mathcal{K})^d$, $\mathcal{K}\supseteq\H$, and the corresponding dual $\mathbf{S}\in\mathfrak{B}(\K\ominus\H)^d$. Then $\mathbf{T}$ is Taylor invertible if and only if $\mathbf{S}$ is Taylor invertible.
	\end{corollary}
	
	\bigskip 
	
	\section*{Declarations}
	\noindent{\bf{Availability of data and materials}}\\
	\noindent No data were used to support this study.
	\vspace{0.5cm}\\
	\noindent{\bf{Conflict of interest}}\\
	\noindent The author declares that he has no conflict of interest.
	\vspace{0.5cm}
	


\vspace{0.5cm}
	
	
	
	
	
	
	
	
	
	
	\begin{thebibliography}{00}
		
		
		\bibitem{Athavale88}
		\textsc{A. Athavale}. \emph{On joint hyponormality of operators}, Proceeding of the
		American Mathematical Society, \textbf{103}(2) (1988), 417--423.
		
		\bibitem{Athavale89}
		\textsc{A. Athavale}. \emph{On the duals of subnormal tuples},
		Integral Equations and Operator Theory, \textbf{12} (1989), 305--325.
		
		\bibitem{AthavalePodder15}
		\textsc{A. Athavale, S. Podder}. \emph{On the reflexivity of certain operator tuples}, Acta Sci. Math. (Szeged), 
		\textbf{81} (2015), 285--291.
		
		\bibitem{BenhidaZerouali07}
		\textsc{C. Benhida, E. H. Zerouali}. \emph{On Taylor and other joint spectra for commuting $n$-tuples of operators}, J. Math. Anal. Appl., \textbf{326} (2007), 521--532.
		
		\bibitem{BenhidaZerouali11}
		\textsc{C. Benhida, E. H. Zerouali}. \emph{Spectral properties of commuting operations for $n$-tuples}. Proc. Amer. Math. Soc., \textbf{139} (2011),  4331--4342.
		
		\bibitem{Brown53}
		\textsc{A. Brown}. \emph{On a class of operators}, Proc. Amer. Math. Soc., \textbf{4} (1953), 723--728.
		
		\bibitem{Campbell75}
		\textsc{S. L. Campbell}. \emph{Subnormal operators with nontrivial quasinormal extensions}, Acta Sci. Math. (Szeged), \textbf{37} (1975), 191--193. 
		
		\bibitem{Conway81}
		\textsc{J. B. Conway}. \emph{The dual of a subnormal operator}, J. Operator Theory, \textbf{5}(2) (1981), 195--211.
		
		\bibitem{Conway91}
		\textsc{J. B. Conway}. \emph{The Theory of Subnormal Operators}, Math. Surveys Monographs, vol. 36. Amer. Math. Soc. Providence, (1991).
		
		\bibitem{ConwayOlin77}
		\textsc{J. B. Conway, R. F. Olin}. \emph{A functional calculus for subnormal operators, II}, Mem. Amer. Math. Soc, no. 184.
		
		\bibitem{Curto80}
		\textsc{R. Curto}. \emph{On the Connectedness of Invertible $n$-tuples}, Indiana University Mathematics Journal, \textbf{29}(3) (1980), 393--406.
		
		\bibitem{Curto81a}
		\textsc{R. Curto}. \emph{Spectral inclusion for doubly commuting subnormal $n$-tuples}, Proceedings of the American Mathematical Society, \textbf{83}(4) (1981), 730--734.
		
		\bibitem{Curto81b}
		\textsc{R. Curto}. \emph{Fredholm and invertible $n$-tuples of operators. The deformation problem}, Trans. Amer. Math. Soc., \textbf{266} (1981), 129--159.
		
		\bibitem{Curto90}
		\textsc{R. Curto}. \emph{Joint hyponormality: a bridge between hyponormality and subnormality}, Proc. Sympos. Pure Math., \textbf{51} (1990), 69--91.
		
		\bibitem{CurtoLeeYoon05}
		\textsc{R. Curto, S. H. Lee, J. Yoon}. \emph{$k$-hyponormality of multivariable weighted shifts}, J. Funct. Anal., \textbf{229} (2005), 462--480.
		
		\bibitem{CurtoLeeYoon07}
		\textsc{R. Curto, S. H. Lee, J. Yoon}. \emph{Hyponormality and subnormality for powers of commuting pairs of subnormal operators}, J. Funct. Anal., \textbf{245} (2007), 390--412.
		
		\bibitem{CurtoLeeYoon20}
		\textsc{R. Curto, S. H. Lee, J.Yoon}. \emph{Quasinormality of powers of commuting pairs of bounded operators}. Journal of Functional Analysis
		\textbf{278}(3) (2020), 108342.
		
		\bibitem{CurtoMuhlyXia88}
		\textsc{R. Curto, P. Muhly, J. Xia}. \emph{Hyponormal pairs of commuting operators}, Oper. Theory Adv. Appl., \textbf{35} (1988), 1--22.
		
		\bibitem{CurtoYoon06}
		\textsc{R. Curto, J. Yoon}. \emph{Jointly hyponormal pairs of subnormal operators need not be jointly subnormal}, Trans. Amer. Math. Soc., \textbf{358} (2006), 5139--5159.
		
		\bibitem{CurtoYoon16}
		\textsc{R. Curto, J. Yoon}. \emph{Toral and spherical Aluthge transforms for 2-variable weighted shifts}, Comptes Rendus Mathematique., \textbf{354}(12) (2016), 1200--1204.
		
		\bibitem{CurtoYoon18}
		\textsc{R. Curto, J. Yoon}. \emph{Aluthge transforms of 2-variable weighted shifts}, Integral Equations Operator Theory, \textbf{90}(52) (2018).
		
		\bibitem{CurtoYoon19}
		\textsc{R. Curto, J. Yoon}. \emph{Spherically Quasinormal Pairs of Commuting Operators}. Analysis of Operators on Function Spaces., (2019), 213--237.
		
		\bibitem{CvetkovicIlicStankovic24}
		\textsc{D. S. Cvetkovic-Ili\' c, H. Stankovi\' c}. \emph{On normal complements}, J. Math. Anal. Appl., \textbf{535}(2) (2024), 128216, \url{https://doi.org/10.1016/j.jmaa.2024.128216}.
		
		\bibitem{Embry73}
		\textsc{M. R. Embry}. \emph{A generalization of the Halmos-Bram criterion for subnormality}, Acta. Sci. Math. (Szeged), \textbf{31} (1973), 61--64.
		
		\bibitem{Fuglede50}
		\textsc{B. Fuglede}. {\emph{A commutativity theorem for normal operators},} Proc. Natl. Acad. Sci., \textbf{36} (1950), 35--40.
		
		\bibitem{Furuta01}	
		\textsc{T. Furuta}. \emph{Invitation to Linear Operators. From Matrices to Bounded Linear Operators on a Hilbert Space}, Taylor and Francis, London and New York, (2001).
		
		\bibitem{Gleason06}
		\textsc{J. Gleason}. \emph{Quasinormality of Toeplitz tuples with analytic symbols}, Houston J. Math., \textbf{32} (2006), 293--298.
		
		\bibitem{JablonskiJungStochel14}
		\textsc{Z. J. Jablonski, I. B. Jung, J. Stochel}. \emph{Unbounded Quasinormal Operators Revisited}. Integr. Equ. Oper. Theory, \textbf{79},  (2014), 135--149, \url{https://doi.org/10.1007/s00020-014-2133-1}.
		
		\bibitem{KimYoon17}
		\textsc{J. Kim, J. Yoon}. \emph{Aluthge transforms and common invariant subspaces for a commuting $n$-tuple of operators}, Integral Equ Oper Theor., \textbf{87} (2017), 245--262.
		
		\bibitem{KimYoon19}
		\textsc{J. Kim, J. Yoon}. \emph{Taylor spectra and common invariant subspaces through the Duggal and generalized Aluthge transforms for commuting $n$--tuples of operators}. J Oper Theor., \textbf{81} (2019), 81--105.
		
		\bibitem{Murphy82}	
		\textsc{G. J. Murphy}. \emph{Self-dual subnormal operators}, Comment. Math. Univ. Carolin., \textbf{23}(3) (1982),  467--473.
		
		\bibitem{Putinar84}
		\textsc{M. Putinar}. \emph{Spectral inclusion for subnormal $n$-tuples}, Proceedings of the American Mathematical Society, \textbf{90}(3) (1984), 405--406.
		
		\bibitem{PietrzyckiStochel21}
		\textsc{P. Pietrzycki, J. Stochel}. \emph{Subnormal $n$-th roots of quasinormal operators are quasinormal}, J. Funct. Anal., \textbf{280} (2021),  109001.
		
		\bibitem{PietrzyckiStochel23}
		\textsc{P. Pietrzycki, J. Stochel}. \emph{On $n$-th roots of bounded and unbounded quasinormal operators}, Annali di Matematica, \textbf{202} (2023), 1313--1333, \url{https://doi.org/10.1007/s10231-022-01281-z}.
		
		\bibitem{Rudin91}
		\textsc{W. Rudin}. \emph{Functional Analysis}, McGraw-Hill, International Editions, (1991).
		
		\bibitem{Saito92}
		\textsc{I. Saito}. \emph{On self-dual subnormal operators}, Integral Equations and Operator Theory, \textbf{15}(5) (1992), 864--878.
		
		\bibitem{Stankovic23a}
		\textsc{H. Stankovi\' c}. \emph{Subnormal $n$-th roots of matricially and spherically quasinormal pairs},  Filomat \textbf{37}(16) (2023), 5325--5331.
		
		\bibitem{Stankovic23b}
		\textsc{H. Stankovi\' c}. \emph{Converse of Fuglede Theorem}, Operators and Matrices, \textbf{17}(3) (2023), 705--714.
		
		\bibitem{Stankovic24}
		\textsc{H. Stankovi\' c}. \emph{Spherical Mean Transform of Operator Pairs}, J. Math. Anal. Appl., \textbf{530}(2) (2024), 127743, \url{https://doi.org/10.1016/j.jmaa.2023.127743}.

        \bibitem{Stankovic24a}
        \textsc{H. Stankovi\' c}. \emph{Spherical Heinz transform of operator tuples}, Zeitschrift f\" ur Analysis und ihre Anwendungen, (2024),
        \url{https://doi.org/10.4171/ZAA/1787}.
		
		\bibitem{Taylor70a}
		\textsc{J. L. Taylor}. \emph{A joint spectrum for several commuting operators}, Jour. Funct. Anal., \textbf{6} (1970), 172--191.
		
		\bibitem{Taylor70b}
		\textsc{J. L. Taylor}. \emph{The analytic functional calculus for several commuting operators}, Acta Math., \textbf{125} (1970), 1--38.
		
		\bibitem{Weidmann80}	
		\textsc{J. Weidmann}. \emph{Linear operators in Hilbert spaces}, Springer-Verlag, Berlin, Heidelberg, New York, (1980).
		
	\end{thebibliography}
\end{document}